\documentclass[]{amsart}

\usepackage{amssymb}
\usepackage{amsfonts}
\usepackage{amsmath}
\usepackage{amsthm}
\usepackage{graphicx}
\usepackage{mathrsfs}
\usepackage{dsfont}
\usepackage{amscd}
\usepackage{multirow}
\usepackage[all]{xy}
\normalfont
\usepackage[T1]{fontenc}
\usepackage{calligra}
\usepackage{verbatim}
\usepackage[usenames,dvipsnames]{color}
\usepackage[colorlinks=true,linkcolor=Blue,citecolor=Violet]{hyperref}
\usepackage{enumerate}

\newcommand{\ninn}{n \in \N}
\newcommand{\ninnz}{n \in \N_0 }
\newcommand{\overn}{\overline{\N}}
\newcommand{\vhi}{{\varphi}}
\newcommand{\N}{{\mathds{N}}}

\newcommand{\R}{{\mathds{R}}}
\newcommand{\C}{{\mathds{C}}}
\newcommand{\cc}{{\mathcal{C}}}

\newcommand{\dom}[1]{{\operatorname*{dom}\left({#1}\right)}}

\theoremstyle{plain}
\newtheorem{theorem}{Theorem}[section]

\newtheorem{corollary}[theorem]{Corollary}

\newtheorem{lemma}[theorem]{Lemma}
\newtheorem{proposition}[theorem]{Proposition}

\newtheorem{theorem-definition}[theorem]{Theorem-Definition}

\theoremstyle{definition}
\newtheorem{definition}[theorem]{Definition}

\theoremstyle{remark}

\renewcommand{\geq}{\geqslant}
\renewcommand{\leq}{\leqslant}

\numberwithin{equation}{section}

\allowdisplaybreaks[4]

\begin{document}

\title[Domains of quantum metrics on AF algebras]{Domains of quantum metrics on AF algebras}
\author{Konrad Aguilar}
\email{konrad.aguilar@pomona.edu}
\urladdr{https://aguilar.sites.pomona.edu}
\address{Department of Mathematics and Statistics, Pomona College, Claremont, CA 91711}
\thanks{The first author gratefully acknowledges the financial support  from the Independent Research Fund Denmark  through the project `Classical and Quantum Distances' (grant no.~9040-00107B)}
\thanks{The first author is supported by NSF grant DMS-2316892}

\author{Katrine von Bornemann Hjelmborg}
\email{kahje19@student.sdu.dk}
\urladdr{}
\address{Department of Mathematics and Computer Science (IMADA) \\ University of Southern Denmark \\  Campusvej 55, DK-5230 Odense M, Denmark}

\author{Fr\'ed\'eric Latr\'emoli\`ere}
\email{frederic.latremoliere@du.edu}
\address{Department of Mathematics, University of Denver, 2199 S. University Blvd, Denver
CO 80208}
\urladdr{http://www.math.du.edu/~frederic/}

\date{\today}
\subjclass[2000]{Primary:  46L89, 46L30, 58B34.}
\keywords{Noncommutative metric geometry, Gromov-Hausdorff convergence, Monge-Kantorovich distance, Quantum Metric Spaces, Lip-norms, Bunce-Deddens algebras, AT-algebras.}

\begin{abstract}
Given a compact quantum metric space (A, L),  we prove that the domain of  L  coincides with A if and only if A is finite dimensional.   We then show how one can explicitly build many   quantum metrics with distinct domains on infinite-dimensional AF algebras. In the last section, we provide a  strategy for calculating the distance between certain states in these quantum metrics, which allow us to calculate the distance between pure states in these quantum metrics on the quantized interval  and on the Cantor space.
\end{abstract}
\maketitle

\tableofcontents

\section{Introduction}

Quantum metric spaces were studied by Rieffel \cite{Rieffel98a, Rieffel99} as the basis for a framework to study various convergence results found, without formal basis, in the mathematical literature. Starting from Connes' original observation that a spectral triple gives rise to a metric on the state space of the underlying C*-algebra \cite{Connes89}, and the realization that this metric is a noncommutative version of the Monge-Kantorovich distance, Rieffel defined a quantum compact metric space in a manner which then enables the construction of analogues of the Gromov-Hausdorff distance, within the noncommutative realm \cite{Rieffel00}. In addition to the initial quantum Gromov-Hausdorff distance introduced by Rieffel, other analogues have since been introduced, with accompanying variation on the notion of a quantum compact metric space, where the underlying "topological" category varies from order-unit spaces \cite{Rieffel00} to C*-algebras \cite{Latremoliere13, Latremoliere13c}. There also exist other useful quantum distances in \cite{Li04, Kerr02,  wu06a}.  Later on, thanks to the functional analytic perspective on the Gromov-Hausdorff distance offered by these new constructions, the third author introduced other distances on a similar model between other structures, such as modules over C*-algebras \cite{Latremoliere16c}, dynamics over C*-algebras \cite{Latremoliere20}, and even spectral triples \cite{Latremoliere22a}.

While convergence of quantum metric spaces has been the focused on the theory of noncommutaitve metric geometry, there is an intrinsic interest in the structure of the noncommutative analogues of the Lipschitz algebras underlying quantum compact metric spaces. For instance in \cite{Latremoliere16b}, the third author proved that a *-morphism between $C^*$-algebras which preserves the domain of the quantum metric structure is automatically Lipschitz for the quantum metrics, in a natural way. It thus appears that the noncommutative Lipschitz algebras may contain some important geometric information by themselves. This paper is written in this spirit: we consider the domain of the quantum analogues of Lipschitz algebras as our focus. More specifically, our paper contains two new results on this structural question.

First, in Section \ref{s:non-equiv-af}, we prove that the domain of the analogue of a Lipschitz seminorm on a unital C*-algebra (see Definition \ref{d:Comquametspa} ) is in fact, the entire C*-algebra, if and only if the C*-algebra is finite dimensional. As a consequence here,  we see that when working with infinite-dimensional quantum compact metric spaces, the domain of these seminorms will be dense, but not the entire space, and thus we can expect to find Lipschitz seminorms with different domains. This is the second matter of this paper, where we study examples of quantum metrics with distinct domains in Section \ref{s:natural-numbers},   which generalizes a result from \cite{Aguilar-Latremoliere15} on the Cantor space.

\begin{definition}\label{d:Statespace}
The state space $S(A)$ of a unital C$^*$-algebra, $A$, is the set of positive linear functionals from $A$ to $\C$, which maps the unit of $A$ to 1. We will denote the norm of $A$ by $\|\cdot\|_A$and its unit by $1_A$.
\end{definition}

The following definition is due to  Rieffel \cite{Rieffel98a} but adapted to our setting.
\begin{definition}[\! {\cite{Rieffel98a}}]\label{d:Comquametspa}
A \textit{compact quantum metric space} $(A, L)$ is a pair of a unital C$^*$-algebra  $A$ and a seminorm $L$ defined on a norm-dense subspace $dom(L)$ of $sa(A)=\{a \in A : a=a^*\}$, where $sa(A)$ is the real subspace of self-adjoint elements of $A$, such that:
\begin{enumerate}
    \item $\{a\in dom(L) : L(a)=0\} = \R1_A $,
    \item the Monge-Kantorovich metric, defined for any two  $\vhi, \psi \in S(A)$ by:
    $$mk_L(\vhi,\psi) = \sup\{|\vhi(a)-\psi(a)| : a\in dom(L) , L(a) \leq 1\},$$
    metrizes the weak$^*$topology on $ S(A)$,
    \item $\{a \in dom(L) : L(a)\leq 1 \}$  is closed in  $\|\cdot \|_A.$
\end{enumerate}
The seminorm $L$ is called a {\em L-seminorm} on $A$.\\
\end{definition}

\begin{definition}[\! {\cite[Definition 1.5.9 and Theorem 1.5.10(Tomiyama)]{Brown-Ozawa}}]\label{d:cond-exp}
Let $A$ be a unital C*-algebra and let $B\subseteq A$ be a unital C*-subalgebra such that $1_A\in B$. A linear map $E : A \rightarrow B$ is a {\em conditional expectation} if $E(b)=b$ for all $b \in B$, $\|E(a)\|_A\leq \|a\|_A$ for all $a \in A$, and $E(bab')=bE(a)b'$ for all $a\in A, b,b'\in B$.

A conditional expectation is {\em faithful} if $E(a^*a)=0$ implies $a=0$.
\end{definition}

\begin{theorem}{{\cite[Theorem 3.5]{Aguilar-Latremoliere15}}}\label{t:af-lip}
Let $A=\overline{\cup_{n \in \N}A_n}^{\|\cdot\|_A}$ be a unital AF-algebra equipped with a faithful tracial state $\tau$ such that $A_0=\C1_A$ and $1_A\in A_n$ for every $n \in \N$. For each $n \in \N$, let 
\[
E_n:A \rightarrow A_n
\]
denote the unique $\tau$-preserving conditional expectation onto $A_n$ (that is, $\tau \circ E_n=\tau$ for all $\ninnz$). Let $(\beta(n))_{n \in \N}$ be a sequence in $(0,\infty)$ that converges to $0$.   

If we define
\[
L_\beta (a)=\sup \left\{ \frac{\|a-E_n(a)\|_A}{\beta(n)} : n \in \N \right\}
\]
for all $a \in A$, then $(A, L)$ and    $(A_n, L)$ for all $n \in \N$ are     compact quantum metric spaces such that 
\[
\Lambda ((A_n, L), (A, L)) \leq \beta(n)
\]
for all $n \in \N$, and thus
\[
\lim_{n \to \infty} \Lambda ((A_n, L), (A, L))=0,
\]
where $\Lambda$ is the quantum Gromov-Hausdorff propinquity of \cite{Latremoliere13}.
\end{theorem}

 \section{Domain and dimension}\label{s:non-equiv-af}
The following two results motivates our inquiry into domains of Lip-norms on AF-algebras. Indeed, in Corollary \ref{c:dom}, we show that the domain of an $L$-seminorm is equal to the the self-adjoint elements of the C*-algebra if and only if the C*-algebra is finite-dimensional.

\begin{theorem}\label{fd-qcms-thm}
  Let $A$ be a Banach space, and let $L$ be a seminorm defined on a dense subspace $\dom{L}$ of $A$ such that $K:= \{a\in\dom{L}:\L(a)=0\}$ is finite dimensional, and the image of $\{a\in\dom{L}:L(a)\leq 1\}$ in $A/K$ is compact in $A$.

  The space $A$ is finite dimensional if, and only if, $\dom{L} = A$.
\end{theorem}

\begin{proof}
  If $A$ is finite dimensional, then the domain $\dom{L}$ of $L$ is also finite dimensional, hence complete, hence closed. Since it is dense in $A$, we conclude that $\dom{L}=A$.

  Conversely, assume that $\dom{L}=A$. Let $B :=A/K$, endowed with the quotient norm form $A$. Let $B_r$ be the image by the canonical surjection of $\{a \in \dom{L} : L(a)\leq r\}$ in $B$, for all $r > 0$. Since $\dom{L} = A$, we observe that $B = \bigcup_{n\in\N} B_n$. By assumption on $L$, the set $B_n$ is compact, hence closed in $B$ for all $\ninn$ (since $B$ is Hausdorff).

  Since $B$ is complete, by the Baire Category Theorem, there exists $\ninn$ such that $B_n$ has nonempty interior. Let $U\subseteq B_n$ be a nonempty open ball. Since $B_n$ is compact, $U$ is totally bounded. Therefore, the open unit ball of $B$ is totally bounded; as $B$ is complete, the closed unit ball of $B$ is thus compact. Consequently, $B$ is finite dimensional. Therefore, $A$ is also finite dimensional as claimed, since $A/K = B$ and $K$ are both finite dimensional.
\end{proof}

\begin{corollary}\label{c:dom}
  If $(A,L)$ is a compact quantum metric space, then $A$ is finite dimensional if, and only if, $\dom{L}=sa(A)$.
\end{corollary}

\begin{proof}
  By Theorem  \ref{fd-qcms-thm}, we conclude that $sa(A) = \dom{L}$ if, and only if, $sa(A)$ is finite dimensional, which in turn is equivalent to $A$ being finite dimensional.
\end{proof}

Therefore, not only do domains of Lip-norms on infinite-dimensional spaces contain the whole space, but  this also tells us that it is possible to obtain Lip-norms with different domains. In the next several results, we see an explicit way to do this for infinite-dimensional AF algebras.

\begin{lemma}\label{l:af-domain}
    Let $A=\overline{\cup_{n \in \N}A_n}^{\|\cdot\|_A}$ be a unital AF-algebra equipped with a faithful tracial state $\tau$ such that $A_0=\C1_A$. For each $n \in \N$, let 
\[
E_n:A \rightarrow A_n
\]
denote the unique $\tau$-preserving conditional expectation onto $A_n$ (that is, $\tau \circ E_n=\tau$ for all $\ninnz$). 

If $a \in A$, then $(\|a-E_n(a)\|_A)_{\ninnz}$ converges to $0.$
    \end{lemma}
    \begin{proof}
    Let $a \in A$. Let $\varepsilon>0$. There exists $N \in \mathds{N}_0$ and $a_N \in A_N$ such that $\|a-a_N\|_A< \varepsilon$. Let $n \geq N$, then 
    \begin{align*}
        \|a-E_n(a)\|_A & \leq \|a-a_N\|_A+\|a_N-E_n(a)\|_A\\
        & <\varepsilon/2+\|E_n(a_N)-E_n(a)\|_A\\
        & =\varepsilon/2+\|E_n(a_N-a)\|_A\\
        & \leq \varepsilon/2+\|a_N-a\|_A\\
        & < \varepsilon/2+\varepsilon/2=\varepsilon. \qedhere
    \end{align*}
    \end{proof}
    
     \begin{theorem}\label{t:af-domain}
    Let $A=\overline{\cup_{n \in \N}A_n}^{\|\cdot\|_A}$ be a unital AF-algebra equipped with a faithful tracial state $\tau$ such that $A_0=\C1_A$. For each $n \in \N$, let 
\[
E_n:A \rightarrow A_n
\]
denote the unique $\tau$-preserving conditional expectation onto $A_n$ (that is, $\tau \circ E_n=\tau$ for all $\ninnz$).  

If $a \in A \setminus \cup_{\ninnz} A_n$, then if we set $\beta_a(n)=(\|a-E_n(a)\|_A)_{\ninnz}$ and $\beta^\infty_a(n)=(\|a-E_n(a)\|^2_A)_{\ninnz}$ for all $\ninnz$, it holds that $(\beta_a(n))_{\ninnz}$ and $(\beta^\infty_a(n))_{\ninnz}$ are sequences in $(0, \infty)$ that converge to $0$ and using notation from Theorem \ref{t:af-lip}
\[
L_{\beta_a}(a)=1\]
and
\[
L_{\beta^\infty_a}(a)=\infty,
\]
and in particular, $a \in \dom{L_{\beta_a}}$ but $a \not\in \dom{L_{\beta^\infty_a}}.$ Moreover, $L_{\beta_a}$ and $L_{\beta^\infty_a}$ are not equivalent.
     \end{theorem}
     \begin{proof}
     Since $a \not\in \cup_{\ninnz} A_n$, we have for all $\ninnz$ that $a\neq E_n(a)$ and so $\|a-E_n(a)\|_A>0$. Hence  $(\beta_a(n))_{\ninnz}$ is a sequence in $(0, \infty)$. By Lemma \ref{l:af-domain}, we have that $(\beta^\infty_a(n))_{\ninnz}$ and $(\beta^\infty_a(n))_{\ninnz}$   converge to $0$. Moreover,
     \[
L_{\beta_a}(a)=\sup_{\ninnz} \frac{\|a-E_n(a)\|_A}{\|a-E_n(a)\|_A}=1
     \]
     and
     \[
     L_{\beta^\infty_a}(a)=\sup_{\ninnz} \frac{\|a-E_n(a)\|_A}{\|a-E_n(a)\|^2_A}=\sup_{\ninnz} \frac{1}{\|a-E_n(a)\|_A}=\infty 
     \]
     as desired.
     \end{proof}
     By Corollary \ref{c:dom}, the previous result holds only holds for any  infinite-dimensional AF-algebra equipped with faithful a tracial state. In particular, this holds for $C(\overn)$, where 
     \[
     \overn=\left\{\frac{1}{2^{n-1}} \in \R : n \in \N\right\}\cup \{0\}
     \]
     is a quantized interval and the one-point compactification of $\N$. 
     Note that $C(\overn)=\overline{\cup_{n \in \N} \cc_n }^{\|\cdot\|_\infty}$ by a Stone-Weierstrass argument, where for every $n \in \N$,
     \begin{equation}\label{eq:cn}
     \cc_n=\{f\in C(\overn): \forall x \leq 1/2^{n-1},f(1/2^{n-1})=f(x) \}.
     \end{equation}
     If we consider the polynomial $p_1$ defined for all $x \in \overline{\N}$ by $p_1(x)=x$, then $p_1 \in C(\overn)\setminus \cup_{\ninnz} \mathcal{C}_n$ and we have by Theorem \ref{t:af-domain} that 
    \[
    L_{\beta^\infty_{p_1}}(p_1)=\infty 
    \]
    but 
\[
L_{\beta_{p_1}}(p_1)=1 
\]
  We also note that  
    \[
    L_{d_1}(p_1)= 1  
    \]
    where $L_{d_1}$ is the usual Lipschitz seminorm defined for every $f \in C(\overn)$ by $L_{d_1}(f)=\sup_{x,y\in \overn} \frac{|f(x)-f(y)|}{|x-y|}$.

\section{Explicit calculations of some quantum metrics}\label{s:natural-numbers}
We continue our study of quantum metrics of Theorem \ref{t:af-lip} on the quantized interval in that we calculate the distance in the quantum metric between the pure states. This was done previously for the Cantor space in [Section 7]\cite{Aguilar-Latremoliere15}. However, we begin with a theorem that generalizes that approach and use this general approach to calculate these distances on the quantized interval.
\begin{theorem}\label{t:state-calc}
Let $A=\overline{\cup_{\ninn}A_n}$ be a unital AF algebra equipped with faithful tracial state $\tau$ with $A_0=\C1_A$. Let $(\beta(n))_{n \in \N}$ be a sequence in $(0,\infty)$ that converges to $0$.

Assume there exists a sequence $(\vhi_n)_{\ninn}$ such that for all $\ninn$
\[
\vhi_n \in sa(A)
\]
and
\[
L_\beta(\vhi_n)=\frac{1}{\beta(n)}.
\]

If $\mu, \nu \in S(A)$ such that $\mu(a)=\nu(a)$ for all $a \in A_n$, then 
\[
\beta(n)|\mu(\vhi_n)-\nu(\vhi_n)|\leq mk_{L_\beta}(\mu, \nu) \leq 2 \beta(n).
\]
In particular, if $\mu(\vhi_n)=-\nu(\vhi_n)$, then 
\[
mk_{L_\beta}(\mu, \nu) = 2 \beta(n).
\]

\end{theorem}
\begin{proof}
Since $L_\beta(\beta(n)\vhi_n)=1$, 
    \begin{align*}
        mk_{L_\beta}(\mu, \nu) &= \sup \{|\mu(a)-\nu(a)| : a\in sa(A), L_\beta(a)\leq 1\}\\
       & \geq |\mu(\beta(n)\vhi_n) - \nu(\beta(n)\vhi_n)|\\
        &=|\beta(n)\mu(\vhi_n) - \beta(n)\nu(\vhi_n)| \\
       & = \beta(n)|\mu(\vhi_n)-\nu(\vhi_n)|
    \end{align*}

    Next, let $a \in A$ such that $L_\beta(a) \leq 1$. Then $\|a-E_n(a)\|_A \leq \beta(n)$. Since $E_n(a) \in A_n$, we have
    \begin{align*}
        |\mu(a)-\nu(a)| &= | \mu(a)- \mu \circ E_n(a) + \nu \circ E_n(a) - \nu(a) |\\
        &\leq |\mu(a-E_n(a))| +|\nu(a-E_n(a))| \\
        &\leq \|a-E_n(a)\|_A + \|a-E_n(a)\|_A\\
        &\leq \beta(n) + \beta(n)\\
        &= 2\beta(n)
    \end{align*} 

    Moreover, if $ \mu(\vhi_n)=-\nu(\vhi)$, then 
    \[
 \beta(n)|\mu(\vhi_n)-\nu(\vhi_n)|=|(\mu-\nu)(\beta(n)\vhi_n)|=2\beta(n),
\]
which implies
\[
mk_{L_\beta}(\mu, \nu) = 2 \beta(n) . \qedhere
\]
\end{proof}

Following the approach in \cite[Section 7]{Aguilar-Latremoliere15}, we find a suitable orthogonal basis for $\cup_{n \in \N} \cc_n$ of Expression \eqref{eq:cn} to serve as the sequence $(\vhi_n)_{n \in \N}$ of Theorem \ref{t:state-calc}. But first, we must find a suitable inner product. We will use a particular faithful state. But first, we produce a method for building   states on $C(\overn)$ by certain elements of $\ell^1$. This might be a well-known result, but we prove it here.
\begin{proposition}\label{p:tracial-state}
    If   $v=(v_n)_{n \in \N} \in \ell^1$ such that $v_n > 0$ for every $n \in \N$ and $\sum_{n=1}^\infty v_n=1$ and 
    \[
    \tau_v(f)=\sum_{n=1}^\infty v_nf(1/2^{n-1})
    \]
    for every $f \in C(\overn)$, then $\tau_v \in S(C(\overn))$ and faithful. In particular, $v_n=\tau_v(\chi_n)$, where $\chi_n=\chi_{\{1/2^{n-1}\}}$, the characteristic function of the set $\{1/2^{n-1}\}$.  
\end{proposition}
\begin{proof}
First, $\tau_v$ is well-defined since $f$ is bounded for every $f \in C(\overn)$. Also, by construction, $\tau_v$ is linear, positive  and $\tau_v(\mathds{1})=1$, where $\mathds{1}$ is the constant one function. Thus $\tau_v \in S(C(\overn))$.  For faithful, assume that $f>0$. If there exists $n \in \N$ such that $f(1/2^{n-1})>0$, then $\tau_v(f)>0$ by construction. Now, assume that $f(0)>0$. Then, by continuity, there exists $N \in \N$ such that $f(1/2^{N-1})>0$ since $\lim_{n\to \infty} f(1/2^{n-1})=f(0)$.  Thus $\tau_v(f)>0.$ Hence, $\tau_v $ is faithful. The proof is complete since $\tau_v(\chi_n)=v_n$ by construction.
\end{proof}
    
 We will use a particular faithful state to induce our inner product. We will find this particular faithful state by first identifying a sequence $(\vhi_n)_{n \in \N}$ that we desire to be an orthogonal basis for $\cup_{n \in \N} \cc_n$ for later calculations involving the conditional expectations of Theorem \ref{t:af-lip}. We will construct this future orthogonal basis from the $\chi_n$ and $\chi_{n_c}$   introduced in the previous proof.

     For $\ninnz$, define
    \[ \vhi_n (x)= \begin{cases} 0 & x>\frac{1}{2^{n-1}}\\
                              -1 & x = \frac{1}{2^{n-1}}\\
                             1 & x < \frac{1}{2^{n-1}}\\
                 \end{cases}\]
                 for all $x \in \overn$. Note that
                 \[
                 \vhi_0=\mathds{1}.
                 \]
                  By construction, 
\begin{equation}\label{eq:phi}
\begin{split}
& \vhi_n =  
       \mathds{1}-\chi_n-\sum_{j=0}^n \chi_j  \quad \text{for all } \ninnz\\
       & \vhi_n=\chi_{n_c}-2\chi_n \quad \text{ for all } \ninn.
       \end{split}
\end{equation} 
where we take $\chi_0=0$. The next lemma helps with prove these form a basis.

\begin{lemma}\label{l:chi-vhi}
For all $\ninnz$, it holds that 
\[
\chi_n=\frac{1}{2^{n+1}}(\vhi_0-2^{n+1}\vhi_n)+\sum_{k=0}^{n}\frac{1}{2^{k+1}}\vhi_{n-k}
\]
\end{lemma}
\begin{proof}
First, for $n=0$, we have
\[
\frac{1}{2}(\vhi_0-2\vhi_0)+\frac{1}{2}\vhi_0=0=\chi_0.
\]
Fix $\ninnz$. Now, assume that \[\chi_j=\frac{1}{2^{j+1}}(\vhi_0-2^{j+1}\vhi_j)+\sum_{k=0}^{j}\frac{1}{2^{k+1}}\vhi_{j-k}\]
 for all $j \in \{0, 1, \ldots, n\}.$ By Expression \eqref{eq:phi}, 
 \begin{align*}
      \chi_{n+1}&=\frac{1}{2}\left(\chi_{(n+1)_c}-\vhi_{n+1}\right)\\
      &=\frac{1}{2}\left(\vhi_0 - \sum_{j=0}^{n}\chi_j -\vhi_{n+1}\right)\\
      &=\frac{1}{2}\left(\vhi_0 - \sum_{j=0}^{n}\left( \left(\frac{1}{2^{j+1}} \vhi_0-\vhi_j\right)+\sum_{k=0}^{j}\frac{1}{2^{k+1}}\vhi_{j-k}\right)  -\vhi_{n+1}\right)\\
      &=\frac{1}{2}\left(\vhi_0 - \sum_{j=0}^{n} \frac{1}{2^{j+1}}\vhi_0+\sum_{j=0}^{n}\vhi_j-\sum_{j=0}^{n}\sum_{k=0}^{j}\frac{1}{2^{k+1}}\vhi_{j-k}  -\vhi_{n+1}\right)\\
      &=\frac{1}{2}\left(\vhi_0 - \sum_{j=0}^{n} \frac{1}{2^{j+1}}\vhi_0+\sum_{j=0}^{n}\vhi_j-\sum_{j=0}^{n}\sum_{k=0}^{j}\frac{1}{2^{k+1}}\vhi_{n-j}  -\vhi_{n+1}\right)\\
      &=\frac{1}{2}\left(\vhi_0 - \sum_{j=0}^{n} \frac{1}{2^{j+1}}\vhi_0+\sum_{j=0}^{n}\vhi_j-\sum_{j=0}^{n}\vhi_{n-j}\left(\sum_{k=0}^{j}\frac{1}{2^{k+1}}\right)  -\vhi_{n+1}\right)\\
      &=\frac{1}{2}\left(\vhi_0 - \sum_{j=0}^{n} \frac{1}{2^{j+1}}\vhi_0+\sum_{j=0}^{n}\vhi_j-\sum_{j=0}^{n}\vhi_{n-j} + \sum_{j=0}^{n}\frac{1}{2^{j+1}}\vhi_{n-j}  -\vhi_{n+1}\right)\\
      & \quad \quad \quad \quad \quad \quad \quad \quad \quad \quad \quad   \quad \quad \quad \quad \quad \quad \quad \quad \text{ since $\sum_{k=0}^{j}\frac{1}{2^{k+1}}=1-2^{-j-1}$} \\
     &=\frac{1}{2}\left(\vhi_0 - \sum_{j=0}^{n} \frac{1}{2^{j+1}}\vhi_0+\sum_{j=0}^{n}\vhi_j-\sum_{j=0}^{n}\vhi_{j} + \sum_{j=0}^{n}\frac{1}{2^{j+1}}\vhi_{n-j}  -\vhi_{n+1}\right)   \\
     &=\frac{1}{2}\left(\frac{1}{2^{n+1}}\vhi_0  + \sum_{j=0}^{n}\frac{1}{2^{j+1}}\vhi_{n-j}  -\vhi_{n+1}\right)   \\
     &=\frac{1}{2^{n+2}}\vhi_0  + \sum_{j=0}^{n}\frac{1}{2^{j+2}}\vhi_{n-j}  -\frac{1}{2}\vhi_{n+1} \\
     &=\frac{1}{2^{n+2}}\vhi_0  + \sum_{j=0}^{n}\frac{1}{2^{j+2}}\vhi_{n-j}  -\frac{1}{2}\vhi_{n+1} +\frac{1}{2}\vhi_{n+1} -\frac{1}{2}\vhi_{n+1} \\
     &=\frac{1}{2^{n+2}}\vhi_0  - \vhi_{n+1} + \sum_{j=-1}^{n}\frac{1}{2^{j+2}}\vhi_{n-j} \\
     &=\frac{1}{2^{n+2}}\left(\vhi_0  - 2^{n+2}\vhi_{n+2}\right) + \sum_{k=0}^{n+1}\frac{1}{2^{k+1}}\vhi_{n+1-k} ,   
 \end{align*}
 which completes the proof.
\end{proof}

Now, we are able to provide a basis for the $\cc_n$ using the $\vhi_n$. 
\begin{proposition}\label{p:vhi-span}
$\cc_n=\mathrm{span}\{\vhi_0, \vhi_1, \ldots, \vhi_{n-1}\}$
for all $\ninn.$
\end{proposition}
\begin{proof} 
Since $\vhi_0=\mathds{1}$, we have that $\cc_1=\mathrm{span}\{\vhi_0\}.$

Now, let $\ninn, n \geq 2.$ For all $k\in \{1, \ldots, n-1\}$, we have that $\vhi_k \in \cc_n$ by Expression \eqref{eq:phi}. Hence
\[
 \vhi_0, \vhi_1, \ldots, \vhi_{n-1} \in  \cc_n
\]
and 
\[
\mathrm{span} \{ \vhi_0, \vhi_1, \ldots, \vhi_{n-1}  \} \subseteq \cc_n.
\]

 Next, by Lemma \ref{l:chi-vhi}, we have that $\chi_n \in \mathrm{span}\{\vhi_0, \vhi_1, \ldots, \vhi_{n-1}\}$. Thus, \[\cc_n \subseteq \mathrm{span}\{\vhi_0, \vhi_1, \ldots, \vhi_{n-1}\}.\qedhere\]
\end{proof}

                 As mentioned above, the purpose of $(\vhi_n)_{n \in \N}$ is to serve as an orthogonal basis. Here we find the correct $(v_n)_{n \in \N}$ of Theorem \ref{p:tracial-state} to accomplish this task.
                 
\begin{proposition}\label{p:best-sequence} 
     The following are equivalent: 
    \begin{enumerate}
        \item $\tau_v (\vhi_n) = 0 $ for all $\ninn$.
        \item $v=(2^{-n})_{\ninn}$.
    \end{enumerate}
     \end{proposition}
    \begin{proof} 
     (1) $\implies$ (2) 
   Let $\ninn$. By Lemma \ref{l:chi-vhi},
   \begin{align*}
       v_n&=\tau_v(\chi_n)\\
       & = \tau_v\left(\frac{1}{2^{n+1}}(\vhi_0-2^{n+1}\vhi_n)+\sum_{k=0}^{n}\frac{1}{2^{k+1}}\vhi_{n-k}\right)\\
       & = \frac{1}{2^{n+1}}(\tau_v(\vhi_0)-2^{n+1}\tau_v(\vhi_n))+\sum_{k=0}^{n}\frac{1}{2^{k+1}}\tau_v(\vhi_{n-k})\\
        & = \frac{1}{2^{n+1}}\tau_v(\vhi_0)+\frac{1}{2^{n+1}}\tau_v(\vhi_0) \quad \text{ by assumption}\\
        & = \frac{1}{2^{n+1}}+\frac{1}{2^{n+1}}\\
        &=\frac{1}{2^n}.
   \end{align*}
   
   (2) $\implies$ (1)  Let $\ninn$. By Expression \eqref{eq:phi}, we have
   \begin{align*}
       \tau_v(\vhi_n)&=\tau_v\left(\mathds{1}-\chi_n-\sum_{j=0}^n \chi_j\right)\\
                     &=\tau_v(\mathds{1})-\tau_v(\chi_n)-\sum_{j=0}^n \tau_v(\chi_j)\\
                     &=1-2^{-n}-\sum_{j=1}^n 2^{-j}\\
                     &=1-2^{-n}-(1-2^{-n})\\
                     &=0 \qedhere.
   \end{align*}
    \end{proof} 
    Thus, for the remainder of the article, we will assume $v=(2^{-n})_{\ninn}$.

    \begin{lemma}\label{l:ortholemma}
    For each $n,m \in \N$, it holds that
    \[
    \vhi_n\vhi_m=\begin{cases}
            \vhi_{\max\{m,n\}} &  : m \neq n\\
            \chi_{n_c} &  :  m=n
    \end{cases}
    \]
    \end{lemma}
    \begin{proof}
    First, assume $m\neq n$, and without loss of generality $m<n$. If $m=0$, then $\vhi_n \vhi_0=\vhi_n$. 
    
    Next, if $m>0$, then by Expression \eqref{eq:phi},
    \begin{align*} 
        \vhi_n \vhi_m &= \bigg(\mathds{1} - \chi_n - \sum_{j=0}^{n}\chi_j\bigg)\bigg(\mathds{1} - \chi_m -                   \sum_{j=0}^{m}\chi_j\bigg)\\
                     &=  \mathds{1} - \chi_m - \sum_{j=0}^{m}\chi_j-\chi_n +0+0-\sum_{j=0}^n \chi_j +\chi_m+\sum_{j=0}^m\chi_j\\
                     &= \mathds{1} -\chi_n-\sum_{j=0}^n \chi_j\\
                     &=\vhi_n. 
    \end{align*}
    
    Second, if $m=n$, we have 
        \begin{align*} 
        \vhi_n \vhi_n &= \bigg(\mathds{1} - \chi_n - \sum_{j=0}^{n}\chi_j\bigg)\bigg(\mathds{1} - \chi_n -                   \sum_{j=0}^{n}\chi_j\bigg)\\
                     &=  \mathds{1} - \chi_n - \sum_{j=0}^{n}\chi_j-\chi_n +\chi_n+\chi_n-\sum_{j=0}^n \chi_j +\chi_n+\sum_{j=0}^n\chi_j\\
                     &=  \mathds{1}  - \sum_{j=0}^{n}\chi_j +\chi_n\\
                     &=\mathds{1}  - \sum_{j=0}^{n-1}\chi_j \\
                     & = \chi_{n_c}.\qedhere
    \end{align*}
    \end{proof}
 
\begin{theorem}\label{t:ortho-basis} 
   For each $\ninn$, the set $\{\vhi_0, \vhi_1, \ldots, \vhi_{n-1}\}$ is an orthogonal basis for $\cc_n$ with respect to $\tau_v$.
\end{theorem}
\begin{proof}
   This follows from Proposition \ref{p:vhi-span}, Proposition \ref{p:best-sequence}, and Lemma \ref{l:ortholemma}.
\end{proof}

  Now we see how our choice of faithful state and orthogonal basis come together.
    
    \begin{lemma}\label{l:}
    Let $(\beta(n))_{\ninn}$ be non-increasing sequence of positive real numbers that converges to $0$. For all $\ninn$, it holds that
    \[
    L_\beta(\vhi_n)=\frac{1}{\beta(n)},
    \]
    where we use $\tau_v$ as our faithful tracial state in Theorem \ref{t:af-lip}.
    \end{lemma}
    \begin{proof}
    Let $\ninn$. Note that $\vhi_n \in \cc_m$ for all $m \geq n+1.$ Hence $E_m(\vhi_n)=\vhi_n$ for all $m \geq n+1$ since $\vhi_n \in \mathcal{C}_m$. Therefore,
    \[
    \frac{\|\vhi_n-E_m(\vhi_n)\|_{C(\overn)}}{\beta(m)}=\frac{\|\vhi_n-\vhi_n\|_{C(\overn)}}{\beta(m)}=0
    \]
    for all $m \geq n+1.$

   Next, if $m \leq n$, then $E_m(\vhi_n)=0$ since $E_m$ is the orthogonal projection onto 
    \[
    \cc_m=\mathrm{span}\{\vhi_0, \vhi_1, \ldots, \vhi_{m-1}\}
    \]
    (see Proposition \ref{p:vhi-span}), 
    and $\vhi_n$ is orthogonal to all $\vhi_0, \vhi_1, \ldots, \vhi_{m-1}$ with respect to $\tau_v$ by Theorem \ref{t:ortho-basis}.  Therefore,
    \[
    \frac{\|\vhi_n-E_m(\vhi_n)\|_{C(\overn)}}{\beta(m)}=\frac{\|\vhi_n-0\|_{C(\overn)}}{\beta(m)}=\frac{\|\vhi_n\|_{C(\overn)}}{\beta(m)}=\frac{1}{\beta(m)}
    \]
    for all $m \leq n.$
    
    Hence 
    \begin{align*}
    L_\beta (\vhi_n)&=\max_{m \in \N, m \leq n} \frac{1}{\beta(m)}=\frac{1}{\beta(n)}.\qedhere
    \end{align*}
    \end{proof}
    
 We are now able to calculate exactly the distance between any two pure states on $C(\overn)$.
    
    \begin{theorem}
     Let $(\beta(n))_{\ninn}$ be non-increasing sequence of positive real numbers that converges to $0$. Using $\tau_v$ as our faithful tracial state in Theorem \ref{t:af-lip}, 
 it holds for every $x,y \in  \overn $ that 
\[
mk_{L_\beta}(\delta_x,\delta_y)=2 \beta(\min\{1-\log_2(y), 1-\log_2(x)\}) 
\]
with the convention that $\log_2(0)=-\infty$,
  where $\delta_x,\delta_y$ are evaluation at $x,y\in \overn$.
    \end{theorem}
    \begin{proof} 
    Let $x,y \in \overn$ and $x<y$. Let $n \in \N$ such that $y=\frac{1}{2^{n-1}}$. Note that 
    \[
    L_\beta (\beta(n) \vhi_n)=1
    \]
    by Lemma \ref{l:}. Let $f \in \cc_n$. Then 
    \[\delta_x(f)=f(x)=f(1/2^{n-1})=f(y)=\delta_y(f).\] 
    Moreover,
    \[
    \delta_y(\vhi_n)=\vhi_n(y)=-1=-\vhi_n(x)=\delta_x(\vhi_n).
    \]
    Thus, 
    \[
    mk_{L_\beta}(\delta_x,\delta_y)=2\beta(n).
    \]
    Since $\log_2(y)=1-n$ and $\log_2(x)<\log_2(y)$, the proof is complete by Theorem \ref{t:state-calc}.
    \end{proof}

We also note that Theorem \ref{t:state-calc}   also recovers \cite[Theorem 7.5]{Aguilar-Latremoliere15}, which calculates distances between pure states on the Cantor space.

    \bibliographystyle{amsplain}
\bibliography{ref}
\vfill
\end{document}